\documentclass[12 pt]{article}

\usepackage{amsmath}
\usepackage{amsfonts}
\usepackage{amssymb}
\usepackage{graphicx}
\usepackage[margin=3cm]{geometry}
\usepackage{hyperref}
\usepackage{enumerate}
\usepackage{mathrsfs}
\usepackage{comment}
\usepackage{caption}
\usepackage{subcaption}

\usepackage{amsthm}
\usepackage{mathtools}
\usepackage{thmtools, thm-restate}

\newcommand{\parenitem}{\stepcounter{enumi}\item[(\theenumi)]}

\newcommand{\R}{\mathbb{R}}
\newcommand{\Q}{\mathbb{Q}}
\newcommand{\Z}{\mathbb{Z}}

\newcommand{\A}{\mathbb{A}}

\DeclareMathOperator{\Homeo}{Homeo}

\DeclareMathOperator{\dist}{dist}

\DeclareMathOperator{\Out}{Out}

\newtheorem{theorem}{Theorem}[section]
\newtheorem{lemma}[theorem]{Lemma}
\newtheorem{proposition}[theorem]{Proposition}
\newtheorem{corollary}[theorem]{Corollary}
\newtheorem{question}[theorem]{Question}

\theoremstyle{definition}
\newtheorem{remark}[theorem]{Remark}

\theoremstyle{definition}
\newtheorem{example}[theorem]{Example}

\theoremstyle{definition}
\newtheorem{definition}[theorem]{Definition}

\usepackage[usenames,dvipsnames]{color}

\begin{document}

\title{When actions of amenable groups can be lifted \\ to the universal cover}
\author{Kiran Parkhe \thanks{The author acknowledges support from the Lady Davis Foundation.}}
\date{\today}
\maketitle

\begin{abstract}
In the first part of this paper, we let $G$ be a finitely-generated amenable group such that $G/[G, G]$ is torsion-free. We suppose that $G$ acts by homeomorphisms homotopic to the identity on a manifold $M$, and give conditions on $M$ which imply that such an action must lift to an action on the universal cover $\tilde{M}$. The circle, all 2-manifolds except the open annulus, and most compact 3-manifolds satisfy these conditions. The proof uses a dynamical tool called homological rotation vectors, and Thurston's Geometrization Theorem in the latter case.

On manifolds not satisfying our conditions, such actions really may fail to lift. In the second part, we try to understand the dynamical possibilities in the simplest case: $G = \Z^2$, and $M = \A$ is the open annulus. We show that if a $\Z^2$ action homotopic to the identity on $\A$ fails to lift to a $\Z^2$ action on the plane, and if the action satisfies one additional condition (which may not be necessary), the action is essentially similar to the one generated by $\bar{f_0}(\theta, y) = (\theta + y, y)$ and $\bar{g_0}(\theta, y) = (\theta, y + 1)$.
\end{abstract}

\section{Introduction}
Let $M$ be a manifold, which in this paper we always take to mean a connected orientable topological manifold (Hausdorff, second-countable, locally Euclidean topological space), possibly with boundary. Let $\pi\colon \tilde{M} \to M$ be the universal cover. If $f\colon M \to M$ is a homeomorphism, a \emph{lift} $\tilde{f}\colon \tilde{M} \to \tilde{M}$ is a homeomorphism such that $\pi \circ \tilde{f} = f \circ \pi$.

There always exist lifts of $f$ to the universal cover. In fact, let $x \in M$ be arbitrary, and let $\tilde{x} \in \pi^{-1}(x)$. If $\tilde{f}$ is to be a lift of $f$, we must have $\tilde{f}(\tilde{x}) \in \pi^{-1}(f(x))$. Conversely, for any choice of $\widetilde{f(x)} \in \pi^{-1}(f(x))$, there is a unique lift $\tilde{f}$ such that $\tilde{f}(\tilde{x}) = \widetilde{f(x)}$.



Let $G$ be a discrete group. An \emph{action} of $G$ on $M$ is a homomorphism $\phi\colon G \to \Homeo(M)$, the group of homeomorphisms of $M$. A \emph{lift} of $\phi$ to the universal cover is an action $\tilde{\phi}\colon G \to \Homeo(\tilde{M})$ such that for every $g \in G$, $\tilde{\phi}(g)$ is a lift of $\phi(g)$.

\newpage

Lifting a group action on $M$ to the universal cover is not always possible. Here we give some illustrative examples where, for simplicity, $M = S^1$.

\begin{example}
Regard the circle as $\R/\Z$, and consider the $\Z^2$ action generated by $$f(x + \Z) = -x + \Z, \hspace{6pt} g(x + \Z) = x + 1/2 + \Z.$$ Note that $f$ and $g$ commute, but no choice of lifts to the line commute.
\end{example}

The trouble in this example comes from the fact that $f$ is orientation-reversing. In what follows, we will consider only group actions on manifolds by homeomorphisms that are \emph{homotopic to the identity}, which on the circle means orientation-preserving.

If $f\colon M \to M$ is homotopic to the identity, there are special lifts of $f$ to the universal cover called \emph{homotopy lifts}, defined as follows. Let $f_t$ be a homotopy with $f_0 = id$ and $f_1 = f$. There is a unique homotopy $\tilde{f}_t$ on $\tilde{M}$ such that $\tilde{f}_0 = id_{\tilde{M}}$, and $\pi \circ \tilde{f}_t = f_t \circ \pi$ for every $t$. The homeomorphism $\tilde{f}_1$ is called a homotopy lift of $f$. In general, there may be multiple homotopy lifts, since there may be multiple non-homotopic homotopies from the identity to $f$. For example, every lift of an orientation-preserving homeomorphism of the circle to the line is a homotopy lift. However, it is easy to see that homotopy lifts commute with Deck transformations, so any two homotopy lifts differ by an element in the center of the group of Deck transformations.

\begin{example}
Let $G = \Z/2\Z$. $G$ acts on the circle by a half-turn $R_{1/2}$, and this action does not lift to an action on the line, since any lift of $R_{1/2}$ has infinite order.
\end{example}

\begin{example}
Let $G = BS(1, 3) = \langle a, b \colon aba^{-1} = b^3\rangle$ be the Baumslag-Solitar group. (Note that $[a, b] = b^2$.) Let $C_1 = \R \cup \{\infty\}$ be the circle. Let $f'(x) = 3x, g'(x) = x + 1$, and $f'(\infty) = g'(\infty) = \infty$; these are homeomorphisms of $C_1$. 

Let $C_2$ be a double cover of $C_1$. Let $f$ and $g$ be lifts of $f'$ and $g'$ to $C_2$ such that $g$ has rotation number 1/2, i.e., interchanges the two lifts $\infty_1$ and $\infty_2$ of $\infty$. Define $\phi(a) = f, \phi(b) = g$. This action of $BS(1, 3)$ on the circle does not lift to the line, since if $\tilde{f}$ and $\tilde{g}$ are any lifts of $f$ and $g$, $[\tilde{f}, \tilde{g}]$ fixes all lifts of $\infty_1$ and $\infty_2$ to the line. On the other hand, $\tilde{g}$ cannot fix these points since $g$ interchanges $\infty_1$ and $\infty_2$. Thus $[\tilde{f}, \tilde{g}] \neq \tilde{g}^2$.
\end{example}

The last two examples both involve torsion -- the first in an obvious way, the second less obviously. Namely, in $G = BS(1, 3)$, the element $b$ descends to a torsion element in $G/[G, G]$, since $b^2 = aba^{-1}b^{-1} \in [G, G]$. It turns out that requiring the group to have torsion-free abelianization is exactly what is needed.

\begin{example}
Let $\Sigma_g$ ($g \geq 2$) be a closed genus-$g$ surface, and let $G = \Gamma_g = \pi_1(\Sigma_g)$. Recall that $\Gamma_g$ has the presentation $$\Gamma_g = \langle a_1, b_1, \ldots, a_g, b_g \colon [a_1, b_1]\cdots[a_g, b_g] = 1\rangle.$$

Observe that $\Gamma_g$ acts by isometries of the Poincar\'{e} hyperbolic disk $\mathbb{D}$; indeed, if we adopt a hyperbolic metric on $\Sigma_g$, then $\widetilde{\Sigma_g}$ with the pullback metric is isometric to $\mathbb{D}$, and $\Gamma_g$ acts on $\widetilde{\Sigma_g}$ by covering translations. This $\Gamma_g$ action extends continuously to an action $\phi$ on the circle at infinity $\partial^\infty \mathbb{D}$. This action does not lift to the line. Indeed, if $\widetilde{\phi(a_i)}, \widetilde{\phi(b_i)}$ are any lifts of the generators to the line, the reader can check that $[\widetilde{\phi(a_1)}, \widetilde{\phi(b_1)}]\cdots [\widetilde{\phi(a_g)}, \widetilde{\phi(b_g)}]$ is a translation by $2g - 2$. See \cite{Ghys} for more information.
\end{example}

In this example, the non-amenability of $G$ causes trouble. When an amenable group acts on a compact manifold, there is an invariant probability measure, and this allows us to bring certain tools to bear.

We show in Section 2 of the paper that, in many cases, avoiding the problems seen in the above examples is sufficient to guarantee that a group action lifts to the universal cover. More precisely:

\begin{restatable}{theorem}{WhichMLifts}
\label{WhichM}
Let $G$ be a finitely-generated amenable group with torsion-free abelianization. If $M$ is any of the following, then a $G$ action $\phi$ on $M$ by homeomorphisms homotopic to the identity must lift to an action $\tilde{\phi}$ on the universal cover $\tilde{M}$ such that $\tilde{\phi}(g)$ is a homotopy lift of $\phi(g)$ for every $g \in G$:

\begin{itemize}

\item Any 1-manifold (i.e., $S^1$ or, trivially, $\R$)

\item Any 2-manifold (compact or not, with or without boundary) except the open annulus

\item Any compact 3-manifold, except a closed 3-manifold of spherical, Nil, or $\widetilde{SL(2, \R)}$ geometry

\item Any $n$-manifold ($n \geq 3$) decomposing as a nontrivial connected sum
\end{itemize}
\end{restatable}

The idea is as follows. Given a generating set $S = \{g_1, \ldots, g_k\}$ for $G$, we would like to show that if $\tilde{\phi}(g_i)$ are homotopy lifts of $\phi(g_i)$, then whenever $g_{i_1}\cdots g_{i_n}$ is a word equal to the identity, $\tilde{\phi}(g_{i_1})\cdots \tilde{\phi}(g_{i_n}) = id_{\tilde{M}}$. At worst, $\tilde{\phi}(g_{i_1})\cdots \tilde{\phi}(g_{i_n})$ will be a nontrivial covering translation, since it is a lift of $\phi(g_{i_1})\cdots \phi(g_{i_n}) = id_M$.

When $M = S^1$, for example, the worry is that $\tilde{\phi}(g_{i_1})\cdots \tilde{\phi}(g_{i_n})$ may be a nontrivial integer translation. This is a hint that we should look at the \emph{translation number} of lifts of elements of our group action. We will show that we can ensure $\tilde{\phi}(g_{i_1})\cdots \tilde{\phi}(g_{i_n})$ has translation number 0, which implies it is the identity as desired.

In general, when $M$ is any compact smooth manifold and $G$ as in the theorem acts by homeomorphisms homotopic to the identity, we will show that $\tilde{\phi}(g_{i_1})\cdots \tilde{\phi}(g_{i_n})$ has \emph{mean homological translation vector} equal to 0. For many classes of manifolds $M$, including those mentioned in the theorem, a covering translation with zero mean homological translation vector is the identity, again giving us the desired result. See the next section for more information.

\vspace{12pt}

On the other hand, if $M$ is not one of the manifolds listed above, the situation can be quite different. Even in the simplest case, where $G = \Z^2$ -- that is, we have two commuting homeomorphisms, homotopic to the identity -- the action can fail to lift. For instance:

\begin{example}
\label{AnnulusExamples}
Let $\A = \R/\Z \times \R$ be the open annulus. Denote points in $\A$ by $(\theta, y)$, and in the universal cover $\R^2$ by $(x, y)$. Let $\bar{f_0}(\theta, y) = (\theta + y, y)$, and $\bar{g_0}(\theta, y) = (\theta, y + 1)$. Note that $\bar{f_0}$ and $\bar{g_0}$ commute. Let $f_0(x, y) = (x + y, y)$, and $g_0(x, y) = (x, y + 1)$. These do \emph{not} commute. Their commutator is $[f_0, g_0](x, y) = (x + 1, y)$, and this would be true regardless of which lifts we chose.

This example can be generalized in the following straightforward way. Let $T^2 = \R^2/\Z^2$ be the torus, and let $\bar{\bar{f}}\colon T^2 \to T^2$ be a homeomorphism isotopic to the linear torus map $\left( \begin{array}{cc}
1 & 1\\
0 & 1 \end{array} \right).$ This can be lifted to $\bar{f}\colon \A \to \A$, which is unique up to integral translations in the $y$-direction. Call such a $\Z^2$ action, generated by $\bar{f}$ and $\bar{g_0}$, \emph{lifted toral}. Note that $\bar{f}$ commutes with $\bar{g_0}$, but their lifts $f, g_0\colon \R^2 \to \R^2$ fail to commute, for the same reason as above.
\end{example}

\begin{question}
\label{LiftedToralQuestion}
Let $\bar{f}, \bar{g} \colon \A \to \A$ be commuting homeomorphisms, homotopic to the identity, of the open annulus. Suppose that lifts $f$ and $g$ to the plane fail to commute. Does it follow that some element $\bar{e}$ of $\langle \bar{f}, \bar{g}\rangle$ is conjugate to $\bar{g_0}$ above?
\end{question}

If there is such an element $\bar{e}$, then up to conjugacy and pre-composing by an automorphism of $\Z^2$, the action is lifted toral. It would be very interesting if this dynamical characterization follows from the assumption that $\bar{f}$ and $\bar{g}$ fail to lift. See \cite{Parkhe} for some topological and dynamical aspects of this problem. In Section 3 of this paper, we show that this conclusion does follow, if we additionally assume that the action has an element satisfying a ``non-intersection condition'':

\begin{restatable}{theorem}{LiftedToralTheorem}
\label{LiftedToralTheorem}

Let $\bar{f}, \bar{g} \colon \A \to \A$ be commuting homeomorphisms, homotopic to the identity, of the open annulus. Suppose that lifts $f$ and $g$ to the plane fail to commute. Suppose there is an element $\bar{e'}$ of $\langle \bar{f}, \bar{g}\rangle$ that is isotopic to the identity, and an essential circle $c \subset \A$, such that $\bar{e'}(c) \cap c = \emptyset$. Then a possibly different element $\bar{e} \in \langle \bar{f}, \bar{g}\rangle$ is conjugate to $\bar{g_0}$.
\end{restatable}

To prove this result, we use standard tools in surface dynamics and topology, notably Carath\'{e}odory's theory of prime ends, and Schoenflies' Theorem.

\begin{definition}
The (discrete) Heisenberg group is $H = \left\{\left( \begin{array}{ccc}
1 & a & c\\
0 & 1 & b\\
0 & 0 & 1 \end{array} \right)\colon a, b, c \in \Z\right\}$. This is generated by the matrices $$X = \left( \begin{array}{ccc}
1 & 1 & 0\\
0 & 1 & 0\\
0 & 0 & 1 \end{array} \right), \hspace{6pt} Y = \left( \begin{array}{ccc}
1 & 0 & 0\\
0 & 1 & 1\\
0 & 0 & 1 \end{array} \right), \hspace{6pt} Z = \left( \begin{array}{ccc}
1 & 0 & 1\\
0 & 1 & 0\\
0 & 0 & 1 \end{array} \right)$$ which obey the relations $Z = [X, Y] = XYX^{-1}Y^{-1}$ and $X$ and $Y$ commute with Z; indeed, $H$ can be described as the group generated by three abstract elements $X, Y,$ and $Z$ satisfying these relations.
\end{definition}

If $\bar{f}$ and $\bar{g}$ are commuting homeomorphisms homotopic to the identity of the open annulus, and $f, g\colon \R^2 \to \R^2$ are lifts, then $h = [f, g]$ is a covering transformation, and $f$ and $g$ commute with $h$. This yields an action of the Heisenberg group on the plane; if $h \neq id$, then it is a faithful action. Conversely, a Heisenberg action on the plane such that the generator of the center $h$ is conjugate to a nontrivial translation yields, by taking the quotient by $h$, a $\Z^2$ action by homeomorphisms homotopic to the identity on the annulus.

If $f$ and $g_0$ are as in Example \ref{AnnulusExamples}, we will call $\langle f, g_0 \rangle \subset \Homeo(\R^2)$ a \emph{lifted toral Heisenberg group}. Thus, Question \ref{LiftedToralQuestion} and Theorem \ref{LiftedToralTheorem} are equivalently about the question: which Heisenberg actions on the plane, with generator of the center a translation, have image conjugate to a lifted toral Heisenberg group?

\vspace{12pt}

Before closing the Introduction, we would like to remark that questions about lifting group actions can be recast in terms of short exact sequences and splitting.

\begin{remark}
Let $\tilde{M}$ be the universal cover of $M$, and let $\widetilde{\Homeo(M)} \subset \Homeo(\tilde{M})$ denote the group of homeomorphisms of $\tilde{M}$ that are lifts of homeomorphisms of $M$. There is a short exact sequence $1 \to \pi_1(M) \to \widetilde{\Homeo(M)} \to \Homeo(M) \to 1$, where $\pi_1(M)$ embeds in $\widetilde{\Homeo(M)}$ as Deck transformations (once we choose a base point $x \in M$.

If we have an action $\phi \colon G \to \Homeo(M)$, we can ``pull back'' this short exact sequence by $\phi$ to yield a short exact sequence $$1 \to \pi_1(M) \xrightarrow{i} K \xrightarrow{p} G \to 1 \hspace{6pt} (\dagger)$$ where $$K \coloneqq \{(g, \widetilde{\phi(g)}) \in G \times \widetilde{\Homeo(M)} \colon \widetilde{\phi(g)} \text{ is a lift of } \phi(g)\}.$$ If $\phi$ is injective, $K$ can be identified with the set of lifts of $\phi(G)$. See \cite{Ghys}, p. 34.

The action $\phi$ can be lifted to an action $\tilde{\phi}$ if and only if the short exact sequence $(\dagger)$ \emph{splits}, meaning there is a homomorphism $s \colon G \to K$ such that $p \circ s = id_G$. In this case, $K$ is isomorphic to a semidirect product of $\pi_1(M)$ and $G$.

Suppose that the action of $G$ on $M$ is by homeomorphisms homotopic to the identity, and $\tilde{\phi}$ is a lift such that for every $g \in G$, $\tilde{\phi}(g)$ is a homotopy lift of $\phi(g)$. Then, in fact, $K \cong G \times \pi_1(M)$, via $G \times \pi_1(M) \ni (g, h) \mapsto (g, h\cdot \tilde{\phi(g)}) \in K$. This is an isomorphism because $\tilde{\phi}(G)$ commutes with the group of Deck transformations.

More generally, suppose we have an action by homeomorphisms homotopic to the identity which does not necessarily lift. We can consider the short exact sequence $$1 \to \{\text{homotopy lifts of } id_M\} \xrightarrow{i} K_{\text{htpy}} \xrightarrow{p} G \to 1 \hspace{6pt} (\ddagger)$$ that sits inside $(\dagger)$, where $$K_{\text{htpy}} \coloneqq \{(g, \widetilde{\phi(g)}) \in G \times \widetilde{\Homeo(M)} \colon \widetilde{\phi(g)} \text{ is a \emph{homotopy} lift of } \phi(g)\}.$$ Since homotopy lifts commute with Deck transformations, $i(\{\text{homotopy lifts of } id_M\}) \subset Z(K_{\text{htpy}})$, so $K_{\text{htpy}}$ is a \emph{central extension} of $G$.

In Section 3, where $M = \A$ is the open annulus and $G = \Z^2$ acts by homeomorphisms homotopic to the identity, $(\dagger)$ and $(\ddagger)$ coincide: they are both the sequence $$1 \to \Z \to K \to \Z^2 \to 1$$ If this sequence does not split (i.e., the action does not lift), $K$ has the Heisenberg group as a finite-index subgroup.
\end{remark}

\vspace{12pt}

\noindent \textbf{Acknowledgements.} It is a pleasure to thank the following people: John Franks for his encouragement, and many stimulating conversations; Kathryn Mann for valuable comments, including a natural interpretation of the non-lifting example we give on the Heisenberg manifold $H(\R)/H(\Z)$; and Fr\'{e}d\'{e}ric Le Roux for pointing out a mistake in the proof of Theorem \ref{LiftedToralTheorem} in a previous draft, suggesting a correction, and reading the corrected version. I would also like to thank the Technion for its kind hospitality, and the Lady Davis Foundation for its generosity.

\section{Manifolds on which amenable actions homotopic to the identity lift}
\subsection{Homological translation vectors}
For the reader's convenience, we give a review of translation and rotation numbers for circle homeomorphisms, and their homological generalization for homeomorphisms of higher-dimensional manifolds.

Let $f\colon S^1 \to S^1$ be an orientation-preserving homeomorphism. Take a lift $\tilde{f}\colon \R \to \R$, and define the \emph{translation number} to be $\tau(\tilde{f}) = \lim_{n \to \infty} \frac{\tilde{f}^n(0)}{n}$. It turns out that this limit must exist (and would be the same if we replaced 0 with a different point). For a different lift $\tilde{f}'$, we will have $\tilde{f}' - \tilde{f} = n \in \Z$ (a constant), so $\tau(\tilde{f}') - \tau(\tilde{f}) \in \Z$, and hence the \emph{rotation number} $\rho(f) \coloneqq \tau(\tilde{f}) + \Z \in \R/\Z$ is independent of the chosen lift. Often, by abuse of notation, we write $\rho(f) = \alpha$ as a shorthand for $\rho(f) = \alpha + \Z$.

Rotation numbers, introduced by Poincar\'{e} \cite{Poincare}, encode a lot of dynamical information for homeomorphisms of the circle. A homeomorphism $f\colon S^1 \to S^1$ has a fixed point (respectively, a periodic orbit) if and only if $\rho(f) = 0$ (respectively, $\rho(f) \in \Q/\Z$). Moreover, $\rho(f)$ is irrational if and only if $f$ is semi-conjugate to the irrational rotation $R_\rho(f)$. 

The rotation number map $\rho\colon \Homeo_+(S^1) \to \R/\Z$ on the group of orientation-preserving homeomorphisms of the circle is \emph{not} a homomorphism. For instance, if $f$ is a homeomorphism whose graph crosses the line $y = x$ transversely, then $f$ will have a fixed point that is preserved under small perturbations, so for $\epsilon$ small we will have $$0 = \rho(R_\epsilon \circ f) \neq \rho(R_\epsilon) + \rho(f) = \epsilon.$$

The crucial point for us is that if $\mu$ is a Borel probability measure on $S^1$, and $\Homeo_{\mu, +}(S^1)$ denotes the orientation-preserving homeomorphisms that preserve $\mu$, then the restriction of $\rho$ to $\Homeo_{\mu, +}(S^1)$ \emph{is} a homomorphism. To see this, observe that if $f$ preserves the measure $\mu$ and $\tilde{f}$ is a lift, then $\tau(\tilde{f}) = \int_{x \in S^1} (\tilde{f}(\tilde{x}) - \tilde{x})d\mu$, where $\tilde{x}$ is some choice of lift of $x$. Then, if $f$ and $g$ both preserve $\mu$, we have

$$\begin{array}{lll}
       \tau(\tilde{f}\tilde{g}) & = & \int_{S^1}(\tilde{f}\tilde{g}(\tilde{x}) - \tilde{x})d\mu \\
       \noalign{\medskip}
       & = & \int_{S^1}(\tilde{f}\tilde{g}(\tilde{x}) - \tilde{g}(\tilde{x}))d\mu + \int_{S^1}(\tilde{g}(\tilde{x}) - \tilde{x})d\mu \\
       \noalign{\medskip}
       & = & \int_{S^1}(\tilde{f}\tilde{g}(\tilde{x}) - \tilde{g}(\tilde{x}))d\mu + \tau(\tilde{g}) \\
       \noalign{\medskip}
       & = & \int_{S^1}(\tilde{f}(\tilde{x}) - \tilde{x})d(g_*\mu) + \tau(\tilde{g}) \text{ by the change of variables formula}\\
       \noalign{\medskip}
       & = & \int_{S^1}(\tilde{f}(\tilde{x}) - \tilde{x})d\mu + \tau(\tilde{g}) \text{ since } \mu \text{ is } g \text{-invariant}\\
       \noalign{\medskip}
       & = & \tau(\tilde{f}) + \tau(\tilde{g}).
     \end{array}$$

Thus $\tau$ is a homomorphism on the group of lifts of elements in $\Homeo_{\mu, +}(S^1)$, and $\rho$ is a homomorphism on $\Homeo_{\mu, +}(S^1)$.

If $G$ is any amenable group acting on the circle, then there is an invariant probability measure (since $S^1$ is compact), so rotation number acts as a homomorphism on the image of $G$.

This theory generalizes somewhat to higher-dimensional manifolds, via \emph{homological translation vectors}, which measure via $H_1(M, \R)$ how orbits ``wind around the holes'' of $M$. Ideas of this form have a long history. Limits in the first homology group were first studied by Schwartzman \cite{Schwartzman}, who considered flows rather than maps, but let the space be any compact metric space, which is more general than what we consider. Studying limits in homology for maps rather than flows goes back at least to Rhodes \cite{Rhodes}; and for homeomorphisms homotopic to the identity, to Franks \cite{Franks2} and Pollicott \cite{Pollicott}.

Let $M$ be a compact triangulable (for example, smooth) manifold, possibly with boundary. Suppose $f\colon M \to M$ is a homeomorphism homotopic to the identity; let $f_t$ ($t \in [0, 1]$) be a homotopy from the identity to $f$. For a given $x \in M$, one can consider the path $t \mapsto f_t(x)$, and indeed for every $t > 0$, the path $\gamma_x(t) = f_{t - n}(f^n(x))$, where $n$ is the integer part of $t$. To understand how the path $\gamma_x(t)$ moves around the holes of $M$, we would like to assign classes $[\gamma_x(t)] \in H_1(M, \R)$, $t > 0$, and study $\lim_{t \to \infty} \frac{[\gamma_x(t)]}{t}$. The difficulty is that the path $\gamma_x$ may never be closed, so we must choose some way of closing it up.

Let $b$ be a base point of $M$. Choose a simplicial decomposition of $M$. For each $d$-simplex $X^d$ in this decomposition ($0 \leq d \leq \dim(M)$), choose an identification between it and the standard $d$-simplex $$X^d_{\text{std}} = \{(x_1, \ldots, x_{d + 1}) \in \R^{d + 1} \colon x_1 + \ldots + x_{d + 1} = 1\}.$$ Note that the standard $d$-simplex has a central point, the point with all coordinates equal to $\frac{1}{d + 1}$. For any other point, one can take a linear path towards this central point. This defines in $X^d \subset M$ a central point and paths from other points in $X^d$ to this central point.

Let $\{c_1, \ldots, c_m\}$ be the set of all central points, including the $0$-simplices. Choose paths $\alpha_i$ from $c_i$ to the base point $b$. For $x \in M$, let $c_{i(x)}$ be the center of the smallest-dimensional simplex containing $x$, and $\beta_x$ the path we have described from $x$ to $c_{i(x)}$. We define $h_n(x, f_t)$ to be a closed loop based at $b$, given by following $\alpha_{i(x)}$ backwards from $b$ to $c_{i(x)}$, then $\beta_x$ backwards from $c_{i(x)}$ to $x$, then $\gamma_x$ for $n$ units of time, then $\beta_{f^n(x)}$ from $f^n(x)$ to $c_{f^n(x)}$, and finally $\alpha_{i(f^n(x))}$ from $c_{f^n(x)}$ to $b$.

If $f_t$ is clear from context, we can simply write $h_n(x)$. Note that this defines a homology class $[h_n(x)] \in H_1(M, \R)$. It follows directly from the definitions that $[h_{n + m}(x)] = [h_n(x)] + [h_m(f^m(x))]$. We claim that the function $M \to H_1(M, \R)$ defined by $[h_1(x)]$ is bounded and Borel measurable. In fact, there are finitely many sets $Y_i$ given by the simplicial decomposition such that $Y_i$ consists of all the points in some $d$-simplex and not in any $(d - 1)$-simplex. Let $Y_{ij} = f(Y_i) \cap Y_j$.

By Urysohn's Theorem, the topology on $M$ is metrizable, say by a metric $d$. Since $M$ is a toplogical manifold with boundary, for any $x \in M$ there is a neighborhood $x \in U_x \subset M$ such that $U_x$ is homeomorphic to $\R^n$ or the upper half space. For $\epsilon$ small enough, the $\epsilon$-ball centered at $x$ will be contained in such a neighborhood $U_x$, and by compactness of $M$ we can find a uniform $\epsilon$ for every $x \in M$. Since the simplices are homeomorphically embedded in $M$, there is an $\epsilon' \leq \epsilon$ such that if $x, y \in Y_{ij}$ and $d(x, y) < \epsilon'$, then there are paths $p_1 \subset f(Y_i) \cap U_x$ and $p_2 \subset Y_j \cap U_x$ from $x$ to $y$, and these paths must be homotopic relative to $\{x, y\}$. It follows from this that the paths $h_1(x)$ and $h_1(y)$ are homotopic. In particular, $[h_1(x)]$ and $[h_1(y)]$ represent the same element in $H_1(M, \R)$.

Let $\{B_k\}$ be a finite collection of $\epsilon'$-balls covering $M$, and $Y_{ijk} = Y_{ij} \cap B_k$. This is a finite collection of Borel sets, on each of which the function $[h_1(x)]$ is constant. Therefore, $[h_1(x)]$ is bounded and Borel measurable.

If $\mu$ is any $f$-invariant Borel probability measure on $M$, then by Birkhoff's ergodic theorem, $$\tau_x(f_t) \coloneqq \lim_{n \to \infty} \frac{[h_n(x)]}{n}$$ exists for $\mu$-almost every $x \in M$, and further, $$\int \tau_x(f_t) d\mu = \int [h_1(x)] d\mu.$$ This quantity is called the \emph{mean homological translation vector of $f_t$ with respect to $\mu$}, denoted $\tau_\mu(f_t)$.

Let $\tilde{f}$ be a homotopy lift of $f$ to $\tilde{M}$. Let $f_t$ and$f_t'$ be two homotopies from $id$ to $f$ that yield the lift $\tilde{f}$. Observe that $f_t$ and $f_t'$ will be homotopic relative to their endpoints $id$ and $f$, so we will have $\tau_x(f_t') = \tau_x(f_t)$ for every $x$ where $\tau_x(f)$ is defined, and $\tau_\mu(f_t') = \tau_\mu(f_t)$. In particular, we have proved the following

\begin{proposition}
\label{ExistMeanTranslation}
Let $M$ be a compact triangulable manifold, possibly with boundary. Let $f\colon M \to M$ be a homeomorphism that is homotopic to the identity, and let $\tilde{f}\colon \tilde{M} \to \tilde{M}$ be a homotopy lift. Let $\mu$ be an $f$-invariant Borel probability measure. Then there is a well-defined mean homological translation vector $\tau_\mu(\tilde{f}) \in H_1(M, \R)$.
\end{proposition}

Note that $\tau_\mu(\tilde{f})$ depends on the choice of homotopy lift. However, if $\{\sigma_1, \ldots, \sigma_n\}$ is a basis of the torsion-free part of $H_1(M, \Z)$, we know $H_1(M, \R)$ can be regarded as a real vector space with this same basis. If we quotient $H_1(M, \R)$ by the $\Z$-span of $\{\sigma_1, \ldots, \sigma_n\}$, then $\tau_\mu(\tilde{f})$ desends to an element $\rho_\mu(f)$ which is independent of the choice of homotopy lift, which we can call the \emph{mean homological rotation vector} of $f$, in analogy with the circle case.

In general, $\tau_x$ and $\tau_\mu$  are less powerful invariants than on the circle: they do not come close to classifying homeomorphisms up to semi-conjugacy. However, they still have significant dynamical implications. Just to name one among many, if $f\colon T^2 \to T^2$ is a homeomorphism isotopic to the identity, and $\tilde{f}$ is a lift of $f$ such that the origin lies in the interior of the convex hull of the set $\{\tau_x(\tilde{f})\colon x \in T^2\}$, then $f$ has a fixed point (each lift of which is fixed by $\tilde{f}$) \cite{Franks3}.

For us, the important property of the mean translation vector is that -- as with translation number for lifts of circle homeomorphisms -- when restricted to the homeomorphisms preserving a fixed probability measure, it is a homomorphism. More precisely, let $\mathcal{G}(M)$ be the group of homotopy lifts of homeomorphisms of $M$ that are homotopic to the identity. Let $\mathcal{G}_\mu(M)$ be the subgroup of homotopy lifts of homeomorphisms that preserve the probability measure $\mu$.


\begin{proposition}
\label{IsHomomorphism}
The map $\tau_\mu \colon \mathcal{G}_\mu(M) \to H_1(M, \R)$ is a homomorphism; that is, if $\tilde{f}$ and $\tilde{g}$ are homotopy lifts of $f$ and $g$, then $$\tau_\mu(\tilde{g}\tilde{f}) = \tau_\mu(\tilde{g}) + \tau_\mu(\tilde{f}).$$
\end{proposition}

The calculation is essentially the same as the one above for the circle; it uses the change of variables formula in integration, together with the fact that $\mu$ is an invariant measure. See \cite{Franks} for details.

\subsection{Proof of Theorem \ref{WhichM}}
\begin{theorem}
\label{CommutingLifts}
Let $M$ be a manifold, possibly with boundary. Suppose $M$ satisfies at least one of the following conditions:

\begin{enumerate}

\parenitem $Z(\pi_1(M))$ is trivial.

\parenitem $M$ is compact and triangulable (for example, smooth), and for every $h \in Z(\pi_1(M))$ such that some nonzero power $h^n$ lies in $[\pi_1(M), \pi_1(M)]$, $h$ is trivial.

\parenitem $M$ has a boundary component satisfying (1) or (2).
\end{enumerate}

If $G$ is any finitely-generated amenable group with torsion-free abelianization, then any action $\phi$ of $G$ on $M$ by homeomorphisms homotopic to the identity lifts to an action $\tilde{\phi}$ on the universal cover $\tilde{M}$ such that $\tilde{\phi}(g)$ is a homotopy lift of $\phi(g)$ for every $g \in G$.
\end{theorem}

\begin{proof}
It is easy to see that if conditions (1) and (2) are sufficient conditions for a lift to exist, then so is condition (3). Indeed, assume the sufficiency of (1) and (2), and suppose $M$ satisfies (3). Let $C$ be a component of $\partial M$ satisfying condition (1) or (2). Let $\pi\colon \tilde{M} \to M$ be the universal cover, and let $\tilde{C}$ be a component of $\pi^{-1}(C)$. The restriction $\pi|_{\tilde{C}}\colon \tilde{C} \to C$ is a universal cover. Since $C$ satisfies (1) or (2), we may lift $\phi|_C$ to an action $\widetilde{\phi|_C}$ on $\tilde{C}$. But for every $g \in G$, $\widetilde{\phi|_C}(g)$ is the restriction to $\tilde{C}$ of a homotopy lift $\tilde{\phi}(g)$ of $\phi(g)$. The homotopy lifts $\tilde{\phi}(g)$ have the correct relations to define an action of $G$, since they obey the correct relations on $\tilde{C}$.

The sufficiency of condition (1) is also elementary. Suppose that (1) holds. We have seen that two homotopy lifts differ by an element of the center of the group of Deck transformations. If $\pi_1(M) \cong \{$Deck transformations$\}$ has trivial center, then homotopy lifts are unique. If we let $\tilde{\phi}(g)$ be the unique homotopy lift of $\phi(g)$ for every $g \in G$, this defines an action of $G$ since for any $g_1, g_2 \in G$, $\tilde{\phi}(g_1g_2)$ and $\tilde{\phi}(g_1)\tilde{\phi}(g_2)$ are both the unique homotopy lift of $\phi(g_1g_2)$.

The key part of the theorem is sufficiency of condition (2). We have assumed that the abelianization $G/[G, G]$ is torsion-free. Also, $G$ is finitely-generated, and hence so is $G/[G, G]$. Therefore, $G/[G, G] \cong \Z^n$ for some $n$. Thus we may choose $g_1, \ldots, g_n \in G$ such that $g_1[G, G], \ldots, g_n[G, G]$ form a basis for $G/[G, G]$. Choose $\tilde{\phi}(g_i)$ to be arbitrary homotopy lifts of $\phi(g_i)$.

Since $M$ is compact and $G$ is amenable, there is a $\phi$-invariant Borel probability measure $\mu$ on $M$. Since $M$ is compact and triangulable, by Proposition \ref{ExistMeanTranslation} there is a well-defined mean homological translation vector $\tau_\mu$, which by Proposition \ref{IsHomomorphism} is a homomorphism on the group of homotopy lifts of homeomorphisms homotopic to the identity on $M$ that preserve $\mu$.

Now for any $g \in [G, G]$, there exists a homotopy lift $\tilde{\phi}(g)$ such that $\tau_\mu(\tilde{\phi}(g)) = 0 \in H_1(M, \R)$. To see this, write $g = [h_1, h_2]\cdots[h_{2k - 1}, h_{2k}]$. Choose arbitrary homotopy lifts $\widetilde{\phi(h_i)}$, and observe that $\tau_\mu([\widetilde{\phi(h_1)}, \widetilde{\phi(h_2)}]\cdots[\widetilde{\phi(h_{2k - 1})}, \widetilde{\phi(h_{2k})}]) = 0$ since $\tau_\mu$ is a homomorphism.

Every $g \in G$ can be uniquely represented as $g = g_1^{m_1}\cdots g_n^{m_n}g_0$, where $g_0 \in [G, G]$. Define $\tilde{\phi}(g) = \tilde{\phi}(g_1)^{m_1}\cdots \tilde{\phi}(g_n)^{m_n}\tilde{\phi}(g_0)$, where $\tau_\mu(\tilde{\phi}(g_0)) = 0$. We claim this defines an action of $G$. Thus, if $h_1, \ldots, h_k$ are elements of $G$ such that $h_1\cdots h_k = id$, we must show that $\tilde{\phi}(h_1)\cdots\tilde{\phi}(h_k) = id_{\tilde{M}}$.

Since $h_1\cdots h_k = id$, in particular the projections of the $h_i$ in $G/[G, G] \cong \Z^n$ sum to 0, so $\tau_\mu(\tilde{\phi}(h_1)\cdots\tilde{\phi}(h_k)) = \tau_\mu(\tilde{\phi}(h_1)) + \ldots + \tau_\mu(\tilde{\phi}(h_k)) = 0$. Thus, $\tilde{\phi}(h_1)\cdots\tilde{\phi}(h_k)$ is a homotopy lift of the identity with mean translation vector 0.

If $f_t\colon M \to M$ is a homotopy corresponding to $\tilde{\phi}(h_1)\cdots\tilde{\phi}(h_k)$ (so $f_0 = f_1 = id_M$), this induces a closed loop $\gamma_x$ at each point $x \in M$. For every $x$, $\tau_x(f_t)$ is simply the element of $H_1(M, \R)$ induced by $\gamma_x$. Also, all the $\gamma_x$ are freely homotopic, and hence homologous, so they all induce the same element of $H_1(M, \R)$. Since $\tau_\mu(f_t) = 0$, we have $\tau_x(f_t) = [\gamma_x] = 0 \in H_1(M, \R)$ for every $x \in M$. The class $[\gamma_x] \in H_1(M, \Z)$ is a torsion element, so for some $k \geq 1$, $k[\gamma_x] = 0 \in H_1(M, \Z)$. Since $H_1(M, \Z) = \pi_1(M)/[\pi_1(M), \pi_1(M)]$, the $k$th power of the homotopy class of $\gamma_x$ in $\pi_1(M, x)$ lies in the commutator subgroup $[\pi_1(M, x), \pi_1(M, x)]$.

It is also true that the homotopy class of $\gamma_x$ lies in the center of $\pi_1(M, x)$: we have observed that homotopy lifts commute with all Deck transformations, so that homotopy lifts of the identity are central in the group of Deck transformations. By assumption, $M$ is such that any central element of $Z(\pi_1(M))$ with a positive power in $[\pi_1(M), \pi_1(M)]$ is trivial; thus, $\gamma_x$ is homotopically trivial, and $\tilde{\phi}(h_1)\cdots\tilde{\phi}(h_k) = id_{\tilde{M}}$.
\end{proof}

\begin{remark}
If $M$ satisfies condition (1), note that \emph{any} group action by homeomorphisms homotopic to the identity lifts to the universal cover. This is the case, for example, for surfaces of negative Euler characteristic.
\end{remark}

As a corollary of Theorem \ref{CommutingLifts}, we can prove Theorem \ref{WhichM} from the Introduction:

\WhichMLifts*

\begin{proof}
$S^1$ satisfies condition (2).

If $S$ is a surface with nonempty boundary, a component of this boundary will be homeomorphic to $\R$ or $S^1$, so the result follows from condition (3). Suppose $S$ is a surface without boundary. If $S$ is non-compact, then $\pi_1(S)$ is isomorphic to a free group, which has trivial center unless it is isomorphic to $\Z$. Up to homeomorphism, the only non-compact surface with fundamental group isomorphic to $\Z$ is the open annulus. For proofs of these facts, see \cite{A&S}. So for non-compact surfaces, the result follows from condition (1).

Among closed surfaces, the result is trivial for $S^2$; the fundamental groups of surfaces of genus $\geq 2$ are explicitly known, and have trivial center, so again the result follows from condition (1). The torus satisfies condition (2).


Now consider the connected sum $M \# N$, where $M$ and $N$ are non-simply-connected $n$-manifolds ($n \geq 3)$. Then $\pi_1(M \# N) \cong \pi_1(M) * \pi_1(N)$ by Van Kampen's Theorem; since $\pi_1(M \# N)$ is a nontrivial free product, it has trivial center, so $M \# N$ satisfies (2).

Finally, consider compact 3-manifolds. If $M$ has nonempty boundary, this boundary will be a closed 2-manifold, so $M$ will satisfy condition (3). Therefore we may restrict our attention to closed 3-manifolds. Our strategy will be to use condition (2) of Theorem \ref{CommutingLifts}. These manifolds are compact and triangulable, so we need only check the condition on the fundamental group. By the above, we need only consider the manifolds that are prime. There is only one closed orientable 3-manifold that is prime and not irreducible, namely $S^2 \times S^1$. This has fundamental group isomorphic to $\Z$, and hence satisfies (2).

We claim that among closed, orientable, irreducible 3-manifolds, the only ones whose fundamental group has non-trivial center are the Seifert-fibered manifolds. To show this, we need the following result, whose proof depends on the Thurston Geometrization Theorem. Note that if $M \subset N$ is a connected submanifold of a 3-manifold $N$ with incompressible boundary, then the map $\pi_1(M) \to \pi_1(N)$ induced by the inclusion is injective, so $\pi_1(M)$ can be regarded as a subgroup of $\pi_1(N)$.

\begin{theorem}[Theorem 3.1 of \cite{AFW}]
Let $N$ be a compact, orientable, irreducible 3-manifold with empty or toroidal boundary. Write $\pi = \pi_1(N)$. Let $g \in \pi$ be non-trivial. If the centralizer $C_\pi(g)$ is non-cyclic, then one of the following holds:

\begin{enumerate}
\parenitem There exists a JSJ torus $T$ and $h \in \pi$ such that $g \in h\pi_1(T)h^{-1}$ and such that $$C_\pi(g) = h\pi_1(T)h^{-1};$$

\parenitem There exists a boundary component $S$ and $h \in \pi$ such that $g \in h\pi_1(S)h^{-1}$ and such that $$C_\pi(g) = h\pi_1(S)h^{-1};$$

\parenitem There exists a Seifert fibered component $M$ and $h \in \pi$ such that $g \in h\pi_1(M)h^{-1}$ and such that $$C_\pi(g) = hC_{\pi_1(M)}(h^{-1}gh)h^{-1}.$$
\end{enumerate}
\end{theorem}

In our case, suppose there is a nontrivial element $g \in Z(\pi_1(N))$; then, in the notation of the theorem, $C_\pi(g) = \pi$. If (1) held, then we would have $\pi \cong \Z^2$, but this does not arise as the fundamental group of a closed 3-manifold (see \cite{AFW}, Table 2). We are assuming that $N$ has no boundary, so we need not worry about (2). Suppose (3) holds. We can let $h = 1$, since $g$ is central in $\pi$. Therefore, the centralizer of $g$ in $\pi_1(M)$ is equal to $C_\pi(g) = \pi$; in particular, $\pi_1(M) = \pi_1(N)$. This implies that $M = N$, by Theorem 2.5 of \cite{AFW}, so $N$ is Seifert-fibered.

Seifert-fibered 3-manifolds always admit one of the Thurston geometries; that geometry can be any except Sol and hyperbolic. We have already discussed the manifold $S^2 \times S^1$, the unique closed orientable manifold with $S^2 \times \R$ geometry. We are left with the spherical, Euclidean, Nil, $\mathbb{H}^2 \times \R$, and $\widetilde{SL(2, \R)}$ geometries.

Suppose $M$ has spherical, Euclidean, or $\mathbb{H}^2 \times \R$ geometry. By Table 1 of \cite{AFW}, there is a finite covering $M'$ of $M$ such that $M'$ satisfies condition (2). As in the proof of Theorem \ref{CommutingLifts}, suppose $g_1, \ldots, g_n$ are generators of $G$ whose projection to $G/[G, G]$ form a basis. Choose homotopies from $\phi(g_i)$ to the identity, and let $\phi(g_i)'$ be the homotopy lifts to $M'$.

This may not yield an action of $G$. However, let $G' \subset G \times \Homeo(M')$ be generated by $\{(g_i, \phi(g_i)')\}$. Let $pr_1$ and $pr_2$ be projection to the 1st or 2nd coordinate, and $\pi$ be given by $\pi(f') = f$ whenever $f$ is a homeomorphism of $M$ and $f'$ is a lift to $M'$. Then $pr_2$ yields an action of $G'$ on $M'$, and $\pi \circ pr_2 = \phi \circ pr_1$.


Note that $\ker(pr_1) = \{\phi(g_{i_1})'\cdots\phi(g_{i_j})' \colon g_{i_1}\cdots g_{i_j} = id_G\}$. This is finite, because $M'$ is a finite covering of $M$, so $G'$ is amenable. Also, every element of $\ker(pr_1)$ is a Deck transformation, and these commute with the homotopy lifts $\phi(g_i)'$, so $G'$ is a central extension of $G$. In effect, it may have some additional torsion elements in its center. However, it is evident that $G'/[G', G']$ is torsion-free. Therefore, by Theorem \ref{CommutingLifts}, the action of $G'$ on $M'$ lifts to an action of $G'$ on the universal cover $\tilde{M}$.

If $G' \neq G$, this means there are finite-order covering transformations for the cover $\tilde{M} \to M$; that is, $\pi_1(M)$ has torsion. If the geometry of $M$ is Euclidean or $\mathbb{H}^2 \times \R$, this is not the case: any aspherical manifold has torsion-free fundamental group. Therefore, in these cases, we have really found a lifted action $\tilde{\phi}$ of $G$ (such that $\tilde{\phi}(g)$ is a homotopy lift of $\phi(g)$ for every $g \in G$). This finishes the proof.

\end{proof}

\subsection{Examples}

We have already seen in Example \ref{AnnulusExamples} that on the open annulus, there is a large class of lifted toral $\Z^2$ actions, homotopic to the identity, which fail to lift to $\Z^2$ actions on the universal cover. Here, we give some other manifolds (not satisfying the conditions of Theorem \ref{CommutingLifts}) for which such examples exist.

\begin{example}[$\R P^3$]
$\R P^3$ has universal cover $S^3$, with one nontrivial Deck transformation: $x \mapsto -x$. Consider the matrices $$A = \left( \begin{array}{cccc}
1 & 0 & 0 & 0\\
0 & -1 & 0 & 0\\
0 & 0 & 1 & 0\\
0 & 0 & 0 & -1\end{array} \right), \hspace{6pt} B = \left( \begin{array}{cccc}
0 & 1 & 0 & 0\\
1 & 0 & 0 & 0\\
0 & 0 & 0 & 1\\
0 & 0 & 1 & 0\end{array} \right).$$ They commute with the matrix $$C = \left( \begin{array}{cccc}
-1 & 0 & 0 & 0\\
0 & -1 & 0 & 0\\
0 & 0 & -1 & 0\\
0 & 0 & 0 & -1\end{array} \right),$$ and indeed can be connected to the identity by a continuous path of matrices commuting with $C$. Thus, the homeomorphisms of $\R P^3$ that they define under projection are isotopic to the identity on $\R P^3$. However, their commutator is $C$, so the lifts of these homeomorphisms to $S^3$ do not commute (and have commutator $x \mapsto -x$). Obviously, as claimed in the proof of Theorem \ref{WhichM}, up to finite index these lifts do commute. They define an action of the group $G' \cong \langle a, b, c\colon [a, b] = c, [a, c] = [b, c] = id, c^2 = id\rangle$.
\end{example}

\begin{example}[$H(\R)/H(\Z)$]
Let $H(\R)$ denote the real Heisenberg group, and $H(\Z) = H$ the integer lattice in $H(\R)$. Let $M$ be the closed three-dimensional nilmanifold $H(\R)/H(\Z)$. The universal cover $\tilde{M}$ is diffeomorphic to $\R^3$. In fact, $M$ can be realized as the quotient of $\R^3$ by the following maps: $$S(x, y, z) = (x + 1, y, z), \hspace{6pt} T(x, y, z) = (x, y + 1, z), \hspace{6pt} U(x, y, z) = (x + y, y, z + 1).$$ These Deck transformations commute with  $$\tilde{j_t}(x, y, z) = (x + tz, y + t, z), \hspace{6pt} \tilde{k_t}(x, y, z) = (x, y, z + t),$$ so these induce isotopies on $M$. Moreover, if we set $\tilde{j} = \tilde{j_1}$ a lift of $j\colon M \to M$ and $\tilde{k} = \tilde{k_1}$ a lift of $k\colon M \to M$, we have $[\tilde{j}, \tilde{k}] = S$. Thus homotopy lifts of $j$ and $k$ do not commute.

Kathryn Mann pointed out that this example can be understood abstractly as follows: make $H(\R)$ act on itself by multiplication on the left. This descends to an action on $H(\R)/H(\Z)$. We can restrict this to an $H(\Z)$ action on $H(\R)$, which descends to a $\Z^2$ action on $H(\R)/H(\Z)$.

We note that this example is closely related to Example \ref{AnnulusExamples}. If we disregard the $y$-coordinate, which we can do because $\tilde{j}$ and $\tilde{k}$ do not depend on the $y$-coordinate except in the $y$-coordinate, we are left with $\tilde{j}'(x, z) = (x + z, z)$ and $\tilde{k}'(x, z) = (x, z + 1)$, which is exactly $f_0$ and $g_0$ respectively.

Observe also that, although homotopy lifts of $j$ and $k$ do not commute, in this case there do \emph{exist} commuting lifts. Namely, since $[T, U] = S^{-1}$, the lifts $T\tilde{j}$ and $U\tilde{k}$ commute.
\end{example}

\begin{question}
Up to conjugacy, what are the possible $\Z^2$ actions by homeomorphisms of $H(\R)/H(\Z)$ homotopic to the identity with non-commuting homotopy lifts?
\end{question}

We can give some simple information in this direction. Suppose that $j$ and $k$ are commuting homeomorphisms homotopic to the identity on $M$, $\tilde{j}$ and $\tilde{k}$ are homotopy lifts, and $\tilde{j}$ has a fixed point $p$. Then $\tilde{j}$ and $\tilde{k}$ must commute. Otherwise, $[\tilde{j}, \tilde{k}] = S^n$ for some $n \neq 0$. Since $\tilde{j}$ commutes with the group $D$ of Deck transformations, the $D$-orbit $D(p)$ is pointwise fixed by $\tilde{j}$.

Now note that $\tilde{k}\tilde{j}\tilde{k}^{-1} = S^{-n}\tilde{j}$. The fixed point set of $\tilde{k}\tilde{j}\tilde{k}^{-1}$ contains $\tilde{k}(D(p)) = D(\tilde{k}(p))$, and the same must hold for $S^{-n}\tilde{j}$. In fact, for every $i$, the fixed point set of $S^{in}\tilde{j}$ contains $D(\tilde{k}^{-i}(p))$. There is a number $N$ large enough so that for any point $q \in \R^3$, $\sup_i \dist(q, D(\tilde{k}^{-i}(p))) < N$. Thus, the ball $B_N(q)$ contains points moved by $\tilde{j}$ a distance $i\cdot n$ units, for every $i$, which implies that $\tilde{j}$ is not continuous, a contradiction.

\vspace{12pt}

The reader can easily check the following facts. Suppose that $\tilde{k}(x, y, z) = (x, y, z + 1)$ as above. Then $\tilde{j}$ commutes with the Deck transformations and satisfies $[\tilde{j}, \tilde{k}] = S$ if and only if it has the form $$\tilde{j}(x, y, z) = (\phi_1(y, z) + x + z, y + 1, \phi_2(y, z) + z),$$ where $\phi_i(y + 1, z) = \phi_i(y, z + 1) = \phi_i(y, z)$ for $i = 1, 2$.

Suppose that $\tilde{j}(x, y, z) = (x + z, y + 1, z)$ as above. Then $\tilde{k}$ commutes with the Deck transformations and satisfies $[\tilde{j}, \tilde{k}] = S$ if and only if it has the form $$\tilde{k}(x, y, z) = (\phi_1(y, z) + x, \phi_2(y, z) + y, z + 1),$$ where $$\phi_1(y + 1, z) = \phi_1(y, z), \phi_1(y, z + 1) = \phi_1(y, z) + \phi_2(y, z)$$ and $$\phi_2(y + 1, z) = \phi_2(y, z + 1) = \phi_2(y, z).$$

\begin{question}
The manifolds with Thurston geometry $\widetilde{SL(2, \R)}$, essentially twisted circle bundles over higher-genus surfaces, do not satisfy the conditions of Theorem \ref{CommutingLifts}. Do they admit commuting homeomorphisms homotopic to the identity with non-commuting homotopy lifts?
\end{question}

\section{$\Z^2$ actions on the annulus and $H$ actions on the plane}
This section is devoted to the proof of Theorem \ref{LiftedToralTheorem} from the Introduction:

\LiftedToralTheorem*

\begin{definition}
Let $\bar{f}\colon \A \to \A$ be a homeomorphism that is isotopic to the identity. We say that $f$ \emph{has the intersection property} if, for every essential circle $c \subset \A$, $\bar{f}(c) \cap c \neq \emptyset$. Otherwise, we say it \emph{has the non-intersection property}.
\end{definition}

Let us fix some notation. As before, we let $f_0, g_0, h_0$ be the homeomorphisms of the plane given by $$f_0(x, y) = (x + y, y), \hspace{6pt} g_0(x, y) = (x, y + 1), \hspace{6pt} h_0(x, y) = (x + 1, y).$$

\noindent We let $\phi\colon H \to \Homeo(\R^2)$ denote a Heisenberg action on the plane, and $$f = \phi(X), \hspace{6pt} g = \phi(Y), \hspace{6pt} h = \phi(Z),$$ and we assume that $h$ is conjugate to a translation. We set $\A = \R^2/h$, and for any homeomorphism commuting with $h$, denote its projection to $\A$ with a bar. We let $G = \phi(H) = \langle f, g\rangle$, and $\bar{G} = \langle \bar{f}, \bar{g}\rangle$, and we seek an element $e \in G$ whose projection $\bar{e} \in \bar{G}$ is conjugate to $\bar{g_0}$.

\vspace{12pt}

We first mention two corollaries of Theorem \ref{LiftedToralTheorem}.

\begin{corollary}
Without loss of generality, we may assume that the elements of $G$ are orientation-preserving, and hence the elements of $\bar{G}$ are isotopic to the identity (i.e., they do not interchange the ends of the annulus).
\end{corollary}

\begin{proof}
Suppose we have proved Theorem \ref{LiftedToralTheorem} in this case, and suppose $\phi$ satisfies the conditions of the theorem. Then at most an index-two subgroup $G' \subset G$ is orientation-preserving. Since $e \in G'$, we conclude that $G'$ is conjugate to a lifted toral Heisenberg group, so in $\bar{G'}$ there is an element sending every essential circle $c \subset \A$ above itself. Such an element cannot commute with any homeomorphism of $\A$ that interchanges the ends, so in fact $G = G'$.
\end{proof}

\begin{corollary}
The same conclusion holds if $h$ is conjugate to a translation and $\bar{f}$ leaves invariant some circle $c \subset \A$, or some properly embedded line $\bar{\ell} \subset \A$ going from one end of the annulus to the other.
\end{corollary}

\begin{proof}
First assume that $\bar{f}$ leaves invariant a circle $c \subset \A$. We claim that $\bar{g}$ must have the non-intersection property. Let $\ell \subset \R^2$ be the preimage under $\pi \colon \R^2 \to \A$ of $c$; it is an $h$-invariant line. Choose lifts $f, g$. Suppose, by way of contradiction, that $c \cap \bar{g}(c) \neq \emptyset$. Then there exists $x \in \ell$ such that $g(x) \in \ell$.

Notice the following: if $I \subset \ell$ denotes the closed subinterval from $x$ to $h(x)$, there exists $M > 0$ such that $d(y, g(y)) < M$ for all $y \in I$, since $g$ is continuous. In fact, since $g$ commutes with $h$, this holds for all $y \in \ell$.

Now $g(x)$ is within $M$ of $x$, so (because $\ell$ is $h$-invariant and $f^n$ commutes with $h$) there is a number $N$ independent of $n$ such that $d(f^n(g(x)), f^n(x)) < N$. Also, $d(g(f^n(x)), f^n(x)) < M$. It follows that $d(f^n(g(x)), g(f^n(x))) < M + N$ independent of $n$. But this contradicts the fact that $f^n(g(x)) = h^n(g(f^n(x)))$.

Now suppose $\bar{f}$ leaves invariant a line $\bar{\ell} \subset \A$ that goes from one end of the annulus to the other. We will show that $\bar{f}$ has the non-intersection property. We claim that $\bar{g}(\bar{\ell})$ must intersect $\bar{\ell}$ (but not equal $\bar{\ell}$). We use this to show that $\bar{f}|_{\bar{\ell}}$ is conjugate to a translation. If it were not, then an intersection point $\bar{x} \in \bar{\ell} \cap \bar{g}(\bar{\ell})$ would have to be fixed, and we show this is impossible.

If $\bar{g}(\bar{\ell})$ were equal or disjoint from $\bar{\ell}$, in the universal cover $\R^2$ we would get $f$-invariant regions which are either invariant or translated by $g$, contradicting the assumption that their commutator is $h$.

Let $\bar{x} \in \bar{\ell} \cap \bar{g}(\bar{\ell})$. Notice that $\bar{f}(\bar{x}) \in \bar{\ell}$, and $\bar{f}(\bar{x}) \in \bar{f}(\bar{g}(\bar{\ell})) = \bar{g}(\bar{f}(\bar{\ell})) = \bar{g}(\bar{\ell})$, so $\bar{f}(\bar{x})$ is also an intersection point of $\bar{\ell}$ and $\bar{g}(\bar{\ell})$. We claim that $\bar{f}(\bar{x}) \neq \bar{x}$. Suppose otherwise. Let $x$ be a lift of $\bar{x}$ to $\R^2$, and let $\ell$ be the lift of $\bar{\ell}$ corresponding to $x$. Let $f$ be the lift of $\bar{f}$ leaving $\ell$ invariant; notice that $f$ leaves invariant all the lifts of $\ell$ and the domains between these lifts. Now $fg(x) = hgf(x) = hg(x)$; thus $f$ moves $g(x)$ one unit to the right, but this means it sends it into another complementary domain, a contradiction.

Thus the intersection point $\bar{x}$ is sent to a different intersection point $\bar{f}(\bar{x})$. If the $\bar{f}^n(\bar{x})$ accumulated on anything, that would have to be a fixed intersection point. It follows that $\bar{f}|_{\bar{\ell}}$ is conjugate to a translation.

Then $\bar{f}|_{\bar{g}(\bar{\ell})}$ is also conjugate to a translation. We can make an essential circle in the annulus by taking a segment of $\bar{g}(\bar{\ell})$ starting at $\bar{x}$ and going up to its next intersection with $\bar{\ell}$, then going back down to $\bar{x}$ along $\bar{\ell}$. Applying $\bar{f}$ twice will send this circle to another disjoint from it, so $\bar{f}^2$ does not have the intersection property, and then neither does $\bar{f}$ (see Lemma \ref{IntersectionProperty} below).
\end{proof}

Suppose $\bar{e}' \in \bar{G}$ has the non-intersection property, as in the statement of the theorem. One might expect that $\bar{e}'$ itself would be conjugate to $\bar{g_0}$. The next two examples illustrate that this is not the case: in both, $\bar{f}$ has the non-intersection property but $f$ is not conjugate to a translation.

\begin{example}
Let $f(x, y) = (x + y, y + l(y)),$ where $l(y) = \frac{\sin(\pi(2y + 1/2)) + 1}{8},$ and $g(x, y) = g_0.$ The commutator of $f$ and $g$ is $h_0$, with which both $f$ and $g$ commute. Moreover, $f$ sends the $x$-axis above itself, but as we iterate it the images are bounded above by $y = 1/2.$
\end{example}

As the next example illustrates, the lines $f^n(x$-axis) may not be bounded above but still have accumulation.

\begin{example}
Let $f(x, y) = (x + y, y + k(x + y) + l(y)),$ where $l$ is as above and $k(x + y) = \sin(4\pi(x + y)) + 1.$ As before, let and $g(x, y) = g_0.$ Again, their commutator is $h_0$, with which they both commute. Now $$f(x, 0) = (x, k(x) + l(0)) = (x, \sin(4\pi x) + 1 + 1/4),$$ so $f$ sends the $x$-axis above itself. The $x$-axis is not moved up to infinity under iteration of $f.$ Indeed, $f^n(3/8, 1/2) = (3/8 + n/2, 1/2).$

However, neither are the images bounded above.

\begin{proof}
Suppose otherwise; suppose that there is $c$ such that $f^n(x$-axis) lies below $y = c$ for all $n.$ Let $\pi_2(p)$ denote the $y$-coordinate of $p$; note that $\pi_2(f(p)) \geq \pi_2(p)$ for all $p.$

Taking the quotient by $(x, y) \mapsto (x + 1/2, y),$ we get an induced diffeomorphism $\bar{f}$ of the cylinder. It has countably many fixed points, $\bar{p_i} = (\bar{3/8}, 1/2 + i).$ For $p \in x$-axis, since $\pi_2(f^n(p)) < c$ for all $n,$ it must be the case that for each $\epsilon > 0,$ there exists $N$ large enough so that $\pi_2(f^{n + 1}(p)) - \pi_2(f^n(p)) < \epsilon$ for all $n \geq N.$ It is possible to choose $\epsilon$ small enough so that the set of points $\bar{x}$ on the cylinder such that $\pi_2(\bar{f}(\bar{x})) - \pi_2(\bar{x}) < \epsilon$ consists of small neighborhoods $N_i$ around the fixed points $\bar{p_i}.$ Therefore, we have $\bar{f}^n(\bar{p}) \to \bar{p_i}$ for some $i$ as $n \to \infty.$

The derivative at these fixed points is $\begin{pmatrix} 1 & 1 \\0 & 1 \\ \end{pmatrix}.$ Without loss of generality, suppose that $\bar{f}^n(\bar{p}) \to \bar{p_0}.$ For any point $\bar{x}$, let $\theta(\bar{x}) = \arctan(\frac{\pi_2(\bar{x}) - 1/2}{\pi_1(\bar{x}) - 3/8})$, where $\pi_1(\bar{x})$ means the first coordinate of the lift of $\bar{x}$ of minimum distance to $(3/8, 1/2)$. For large enough $N,$ for $n \geq N,$ $\bar{f}^n(\bar{p})$ will stay close enough to $\bar{p_0}$ that the action of $\bar{f}$ on $\theta(\bar{f}^n(\bar{p}))$ will be very close to what the derivative $\begin{pmatrix} 1 & 1 \\0 & 1 \\ \end{pmatrix}$ would do. Therefore, for some even larger $n,$ $\theta$ becomes positive, which means that $\pi_1(\bar{f}^n(\bar{p})) < 3/8.$ But as long as we stay close enough to $\bar{p_0}$, $\pi_1(\bar{f}^n(\bar{p}))$ only decreases further as we increase $n$, so it will never get back to $3/8,$ which means $\bar{f}^n(\bar{p}) \not\to \bar{p_0}$, a contradiction.
\end{proof}
\end{example}

Before we begin proving Theorem \ref{LiftedToralTheorem}, we need a tool called Carath\'{e}odory's \emph{prime end theory}. A good exposition can be found in \cite{Mather}; see also \cite{Parkhe}.

Let $S$ be a connected surface without boundary, and $U \subset S$ a subset which is open, connected, relatively compact, and ``homologically finite'': $H_1(U, \R)$ is finite-dimensional. Naturally associated to $U$, there is a compact surface with boundary $\hat{U}$ and a topological embedding $\omega\colon U \to \hat{U}$ such that $\omega(U)$ is dense in $\hat{U}$. We will identify $U$ with its image under $\omega$. The points in $\hat{U}$ are called \emph{prime points}, the points in $\hat{U} \setminus U$ are called \emph{prime ends}, and $\hat{U}$ is called the \emph{prime end compactification} of $U$. Roughly speaking, in places where $U \subset S$ has a one-point hole, $\hat{U}$ fills in that hole, and in places where $U \subset S$ has a hole with more than one point removed, $\hat{U}$ adds a boundary circle to $U$.

\begin{proposition}[\cite{Parkhe}, Corollary 5.17]
If $U \subset \A$ is homotopic to an essential circle $c \subset \A$, then $\hat{U}$ is homeomorphic to the closed annulus $S^1 \times [0, 1]$.
\end{proposition}

Let $h\colon \R^2 \to \R^2$ be a translation. Given $U \subset \R^2$ which is simply connected, $h$-invariant, and bounded above and below, we can define the prime end ``compactification'' of $U$, even though $U$ is not relatively compact in the plane. Namely, the image $\bar{U} \subset \A$ is relatively compact; we can form its prime end compactification $\hat{\bar{U}}$, which by the proposition is homeomorphic to $S^1 \times [0, 1]$. We can then consider the universal cover $\tilde{\hat{\bar{U}}}$, which is homeomorphic to $\R \times [0, 1]$, and call it the prime end compactification of $U$.

\begin{proposition}[\cite{Parkhe}, Proposition 5.19]
Given a homeomorphism $f\colon S \to S$ such that $f(U) = U$ (where $U$ as above is open, connected, relatively compact, and homologically finite), there is a unique map $\hat{f}\colon \hat{U} \to \hat{U}$ such that $\hat{f}|_U = f$.
\end{proposition}

\begin{corollary}
The operation $f \mapsto \hat{f}$ respects group structure. That is, if two maps $f, g\colon S \to S$ both leave $U$ invariant, then $\widehat{f \circ g} = \hat{f} \circ \hat{g}$.
\end{corollary}

\begin{proof}
This follows from the ``uniqueness'' part of the proposition. Since there is only one extension of $(f \circ g)|_U$ to $\hat{U}$, it must be $\hat{f} \circ \hat{g}$.
\end{proof}

To help us prove Theorem \ref{LiftedToralTheorem}, we need the following lemmas.

\begin{lemma}
\label{InvariantDomain}
Let $V$ be an $f$- and $g$-invariant nonempty simply-connected domain in $\R^2$. Then $V = \R^2.$
\end{lemma}

\begin{proof}
Suppose otherwise. $V$ has an upper or lower frontier (or both); suppose it has an upper frontier, without loss of generality. Taking prime ends, we get an upper boundary line $\ell$. We get induced homeomorphisms $\hat{f}$ and $\hat{g}$ on the space of prime points, whose restriction to the upper boundary line gives us an action of the Heisenberg group on this line with commutator $\hat{h}$ horizontal translation by 1. Notice that $\hat{f}$ and $\hat{g}$ are non-commuting lifts of commuting homeomorphisms homotopic to the identity of the circle $\ell/\hat{h}$, which is impossible by Theorem \ref{WhichM}.
\end{proof}

The following is Lemma 9 of  \cite{Wang}.

\begin{lemma}
\label{IntersectionProperty}
If some homeomorphism $\bar{f}\colon \A \to \A$ has the intersection property, then so does $\bar{f}^n$ for all $n \in \Z$.
\end{lemma}

We claim that there is no loss of generality in assuming that $\bar{f}$ has the non-intersection property. By hypothesis, $\bar{e} = \bar{f}^m\bar{g}^n$ has the non-intersection property. First, we can assume that $m$ and $n$ are relatively prime. If there were some $a, b, c$ such that $m = ca$ and $n = cb$, then $\bar{f}^m\bar{g}^n = (\bar{f}^a\bar{g}^b)^c$, and by Lemma \ref{IntersectionProperty}, $\bar{f}^a\bar{g}^b$ also has the non-intersection property.

Given that $m$ and $n$ are relatively prime, by elementary number theory it is possible to find $a$ and $b$ such that $am + bn = 1$. Then the matrix $\begin{pmatrix}
m & n \\
-b & a \end{pmatrix}$ has determinant 1, so it lies in $SL(2, \Z)$, and hence $(m, n)$ and $(-b, a)$ generate all of $\Z^2$. This implies that $f^mg^n$ and $f^{-b}g^a$ generate $fh^p$ and $gh^q$ for some $p$ and $q$. Since $[fh^p, gh^q] = h$, then, they also generate $f$ and $g$, so in fact $\langle f^mg^n, f^{-b}g^a\rangle = \langle f, g\rangle$; they differ by an automorphism of the Heisenberg group. Changing the given action by this automorphism, we may assume that $\bar{f}$ has the non-intersection property.

We will assume without loss of generality, conjugating if necessary, that $h = h_0$ is the horizontal translation by one. Since $\bar{f}$ has the non-intersection property, without loss of generality $f(x$-axis) $\cap$ $x$-axis $= \emptyset$.

\subsection{Proof of Theorem \ref{LiftedToralTheorem}}
We now begin the proof. Our goal will be to find some element $e$ of the action such that $e(x$-axis) $\cap$ $x$-axis $= \emptyset$ and the images $e^n(x$-axis) go to positive and negative infinity with no accumulation. Then $\bar{e}$ will be conjugate to $\bar{g_0}$, as desired.


Let $\mathbb{H}$ be the open lower half plane. Let

$$U = \bigcup_{n \in \Z} f^n(\mathbb{H}) \cap \bigcup_{n \in \Z} f^n(\mathbb{H}^c).$$

\noindent Equivalently, if we let $U_0$ be the open region between the $x$-axis and $f(x$-axis), then $U = \cup_{n \in \Z} f^n(U_0)$. 

If $U = \R^2$, then the images of the $x$-axis under $f$ are disjoint and go to positive and negative infinity with no accumulation, so we are done.

Therefore, we may assume that $U \subsetneq \R^2$. The proof will be done by considering how $g(U)$ may intersect $U$. There are three possibilities: (1) one contains the other; (2) they are disjoint; or (3) they are not disjoint and neither contains the other. We will show that case (1) results in a contradiction, while cases (2) and (3) lead to the desired result. Case (3) is the hard part of the proof.

\vspace{12pt}
\noindent
(1) Suppose that $U \subseteq g(U)$ or $g(U) \subseteq U.$ Without loss of generality, $U \subseteq g(U).$ Since the $x$-axis is contained in $U$, it is contained in $g(U)$, so $g^{-1}(x$-axis) is contained in $U$. After taking the quotient by $h$, $g^{-1}(x$-axis) is an embedded homologically nontrivial circle, so under iteration of $f$, $g^{-1}(x$-axis) moves toward the upper frontier of $U$. In particular, $$\bigcup_{n \in \Z}f^n(g^{-1}(\mathbb{H})) = \bigcup_{n \in \Z}f^n(\mathbb{H}).$$ It follows that $$g^{-1}(\bigcup_{n \in \Z}f^n(\mathbb{H})) = \bigcup_{n \in \Z}f^n(\mathbb{H}),$$ since $f$ and $g$ commute up to a horizontal translation and $\mathbb{H}$ is invariant under horizontal translation. Thus $$g(\bigcup_{n \in \Z}f^n(\mathbb{H})) = \bigcup_{n \in \Z}f^n(\mathbb{H}).$$ By similar reasoning, $$g(\bigcup_{n \in \Z}f^n(\mathbb{H}^c)) = \bigcup_{n \in \Z}f^n(\mathbb{H}^c),$$ so $g(U) = U.$ But this is impossible, by Lemma \ref{InvariantDomain}.

\vspace{12pt}
\noindent
(2) Suppose that $U$ and $g(U)$ are disjoint. Note that in this case $U$ must be bounded above and below. The $x$-axis and $g(x$-axis) must be disjoint. Also, $\bigcup_{n \in \Z}g^n(U) = \R^2$ by Lemma \ref{InvariantDomain}, since it is $f$- and $g$-invariant. It follows that the lines $g^n(x$-axis) have no accumulation, so we may conjugate $g$ to vertical translation while keeping $h$ as horizontal translation, giving us the desired result.

\vspace{12pt}
\noindent
(3) Suppose that $U$ and $g(U)$ are not disjoint, and also neither contains the other. There may be three regions: the region $X$ above $U$, $U$ itself, and the region $Y$ below $U$. Consider $g(x$-axis); \emph{a priori} it may intersect $X$; $X$ and $U$; $X, U,$ and $Y$; $U$; $U$ and $Y$; or $Y.$ However, it cannot only intersect $U$ since in that case we would have $g(U) = U$. And we have ruled out the cases where it only intersects $X$ or $Y,$ since then $U$ and $g(U)$ would be disjoint. Intersecting $X$ and $U$ or $U$ and $Y$ are essentially the same. So without loss of generality we must deal with the case (a) where $g(x$-axis) intersects $X$ and $U,$ and the case (b) where it intersects $X, U,$ and $Y.$ We will show that (a) leads to the desired result, and (b) leads to a contradiction.

\vspace{12pt}
\noindent
(a) Suppose $g(x$-axis) intersects $X$ and $U$ (but not $Y$). We can find a curve $\gamma \subset U$ lying strictly below $g(x$-axis) whose quotient in $\A = \R^2/h$ is essential. For sufficiently large $n$, $f^n(\gamma)$ lies above the $x$-axis, and hence $f^n(g(x$-axis)) lies above the $x$-axis. Since $f$ and $f^ng$ generate the same Heisenberg group as $f$ and $g,$ we may let $f^ng$ be the new $g,$ so without loss of generality $g(x$-axis) lies above the $x$-axis.

By Lemma \ref{InvariantDomain}, $\bigcup_{m, n \in \Z}f^mg^n(\mathbb{H}) = \R^2,$ where $\mathbb{H}$ is the open lower half plane. Let $\ell$ be a horizontal line above the $x$-axis in $\R^2,$ and let $C \subset \R^2$ be the region lying between the $x$-axis and $\ell$ (inclusive). Since $\{f^mg^n(\mathbb{H})\colon m, n \in \Z\}$ is an open cover of $C,$ and $C/h \subset \A$ is compact, there is a finite sub-cover $\{f^{m_1}g^{n_1}(\mathbb{H}), \ldots, f^{m_k}g^{n_k}(\mathbb{H})\}.$

We claim that if $m' \geq m$ and $n' \geq n,$ with at least one a strict inequality, then $f^{m'}g^{n'}(x$-axis) lies above $f^mg^n(x$-axis). For we have

$$\begin{array}{lll}
       f^{m'}g^{n'}(x\text{-axis}) & = & f^mg^nf^{m' - m}g^{n' - n}h^{(m' - m)(n' - n)}(x\text{-axis})\\
       \noalign{\medskip}
       & = & f^mg^nf^{m' - m}g^{n' - n}(x\text{-axis}).
     \end{array}$$

\noindent Since both $f$ and $g$ send the $x$-axis above itself, $f^{m' - m}g^{n' - n}(x$-axis) lies above the $x\text{-axis}$, hence $f^{m'}g^{n'}(x$-axis) $= f^mg^nf^{m' - m}g^{n' - n}(x$-axis) lies above $f^mg^n(x$-axis). Another way of saying this is that $f^{m'}g^{n'}(\mathbb{H}) \supset f^mg^n(\mathbb{H}).$

It follows that if we let $m = \max\{m_1, \ldots, m_k\}$ and $n = \max\{n_1, \ldots, n_k\},$ then $f^mg^n(\mathbb{H}) \supset C.$ Thus $f^mg^n(x$-axis) lies above $\ell.$ In particular, if $N = \max\{m, n\},$ then $f^Ng^N(x$-axis) lies above $\ell.$ Thus a high enough power of $fg$ sends the $x$-axis above $\ell$, where $\ell$ is arbitrarily high. A parallel argument, replacing the lower half plane with the upper half plane, shows that a sufficiently high negative power of $fg$ sends the $x$-axis below $\ell$ where $\ell$ is arbitrarily far below the $x$-axis.

Therefore, the map $fg$ has the desired properties. Note that this is really $f^{n + 1}g$, because above we changed $f^ng$ to $g$.

\vspace{12pt}
\noindent
(b) Suppose $g(x$-axis) intersects all three regions $X, U,$ and $Y.$ Let $x_1$ and $x_2$ be points in $g(x$-axis) on the lower and upper frontiers of $U$, respectively. Note that these frontiers are $f$-invariant sets, so $f(x_1)$ and $f(x_2)$ are also on the lower and upper frontiers of $U$, respectively. By this reasoning, $f^n(g(x$-axis)) $= g(f^n(x$-axis)) intersects $X, U,$ and $Y$ for every $n$, from which it follows that for any curve $\gamma \subset U$ whose image in $\R^2/h$ is simple, closed, and essential, $g(\gamma)$ intersects $X, U,$ and $Y$.

We will work in $\A = \R^2/h$. Let $\overline{x\text{-axis}} \subset \A$ be the quotient of the $x$-axis. Let $\bar{U} \subset \A$ be the quotient of $U$ by $h$.

\begin{definition}
We say that a curve $c \colon [0, 1] \to \A$ \emph{crosses} $\bar{U}$ if for some $t_1, t_2 \in [0, 1]$, $c(t_1)$ lies in the lower frontier of $\bar{U}$ and $c(t_2)$ lies in the upper frontier of $\bar{U}$. Letting $I$ be the interval between $t_1$ and $t_2$, exclusive, we call $c(I)$ a \emph{crossing} of $\bar{U}$.
\end{definition}

We claim that it is possible to find a component $\bar{V}$ of $\bar{g}(\bar{U}) \cap \bar{U}$ such that for any essential curve $c \subset \bar{U}$, $\bar{g}(c)$ crosses $\bar{U}$ in $\bar{V}_i$. Once we have shown this, we will have done most of the work towards getting a contradiction in case (3)(b), and finishing the proof of Theorem \ref{LiftedToralTheorem}.

Notice that by compactness of $\overline{x\text{-axis}}$ and continuity of $\bar{g}$, $\bar{g}(\overline{x\text{-axis}})$ can only cross $\bar{U}$ a finite number of times. Indeed, one can choose $\epsilon$ small enough so that an $\epsilon$-neighborhood of $\overline{x\text{-axis}}$ is contained in $\bar{U}$, and choose $\delta$ small enough so that whenever $\bar{x}, \bar{y} \in \overline{x\text{-axis}}$ and $d(\bar{x}, \bar{y}) < \delta$, then $d(\bar{g}(\bar{x}), \bar{g}(\bar{y})) < \epsilon$. Then the preimage under $\bar{g}$ of a crossing of $\bar{g}(\overline{x\text{-axis}})$ through $U$ is a subinterval of $\overline{x\text{-axis}}$ of length at least $\delta$, and there can only be finitely many of these. Therefore, there are finitely many connected components of $\bar{g}(\bar{U}) \cap \bar{U}$ containing a crossing of $\bar{g}(\overline{x\text{-axis}})$ through $\bar{U}$.

Let us denote by $s_1, \ldots, s_m \subset \bar{g}(\overline{x\text{-axis}}) \cap \bar{U}$ the crossings of $\bar{g}(\overline{x\text{-axis}})$ through $\bar{U}$, and by $\bar{V_1}, \ldots, \bar{V_n}$ the connected components of $\bar{g}(\bar{U}) \cap \bar{U}$ containing such a crossing, both numbered according to the cyclic arrangement of these sets around $\bar{U}$. Choose the numbering so that $\bar{V}_1$ contains $s_1, \ldots, s_{i_1}$, $\bar{V}_2$ contains $s_{i_1 + 1}, \ldots, s_{i_1 + i_2}$, and so on, up to $\bar{V}_n$ contains $s_{i_1 + \ldots + i_{n - 1} + 1}, \ldots, s_m$.



In Lemma \ref{NoCrossing}, we need the following extension of Schoenflies' Theorem. It follows straightforwardly from a theorem of Homma \cite{Homma}. We thank Fr\'{e}d\'{e}ric Le Roux for pointing it out to us.

\begin{theorem}
Let $\mathcal{F}, \mathcal{F}'$ be two locally finite families of pairwise disjoint topological oriented lines in the plane. Assume that for each $F \in \mathcal{F}$, there exists some $F' \in \mathcal{F}'$ and some orientation-preserving homeomorphism $\Phi_F\colon F \to F'$, in such a way that the map $F \mapsto \Phi_F(F)$ is a bijection between $\mathcal{F}$ and $\mathcal{F}'$. Assume that the correspondence $F \mapsto \Phi_F(F)$ \emph{preserves the combinatorics}: for every $F_1, F_2 \in \mathcal{F}$, if $F_2$ is on the right-hand side of $F_1$, then $\Phi_{F_2}(F_2)$ is on the right-hand side of $\Phi_{F_1}(F_1)$.

Then there exists an orientation-preserving homeomorphism $\Phi\colon \R^2 \to \R^2$ such that for every $F \in \mathcal{F}$, $\Phi|_F = \Phi_F$.
\end{theorem}

\begin{lemma}
\label{NoCrossing}
The crossings $s_i$ and $s_{i + 1}$ are oriented in opposite directions: if one goes from $Y$ to $X$, then the other goes from $X$ to $Y$.

Furthermore, if the crossing $s_i$ goes from the lower frontier of $\bar{U}$ to the upper frontier of $\bar{U}$ (and $s_{i + 1}$ goes from the upper frontier to the lower frontier), then for any essential curve $c$ above $\overline{x\text{-axis}}$, $\bar{g}(c)$ cannot cross $\bar{U}$ between $s_i$ and $s_{i + 1}$. If $s_i$ goes from the upper frontier of $\bar{U}$ to the lower frontier of $\bar{U}$, then for any essential curve $c$ below $\overline{x\text{-axis}}$, $\bar{g}(c)$ cannot cross $\bar{U}$ between $s_i$ and $s_{i + 1}$.
\end{lemma}

The ideas of the proof are illustrated in Figure 1.

\begin{figure}
\centering
\includegraphics[width=3in]{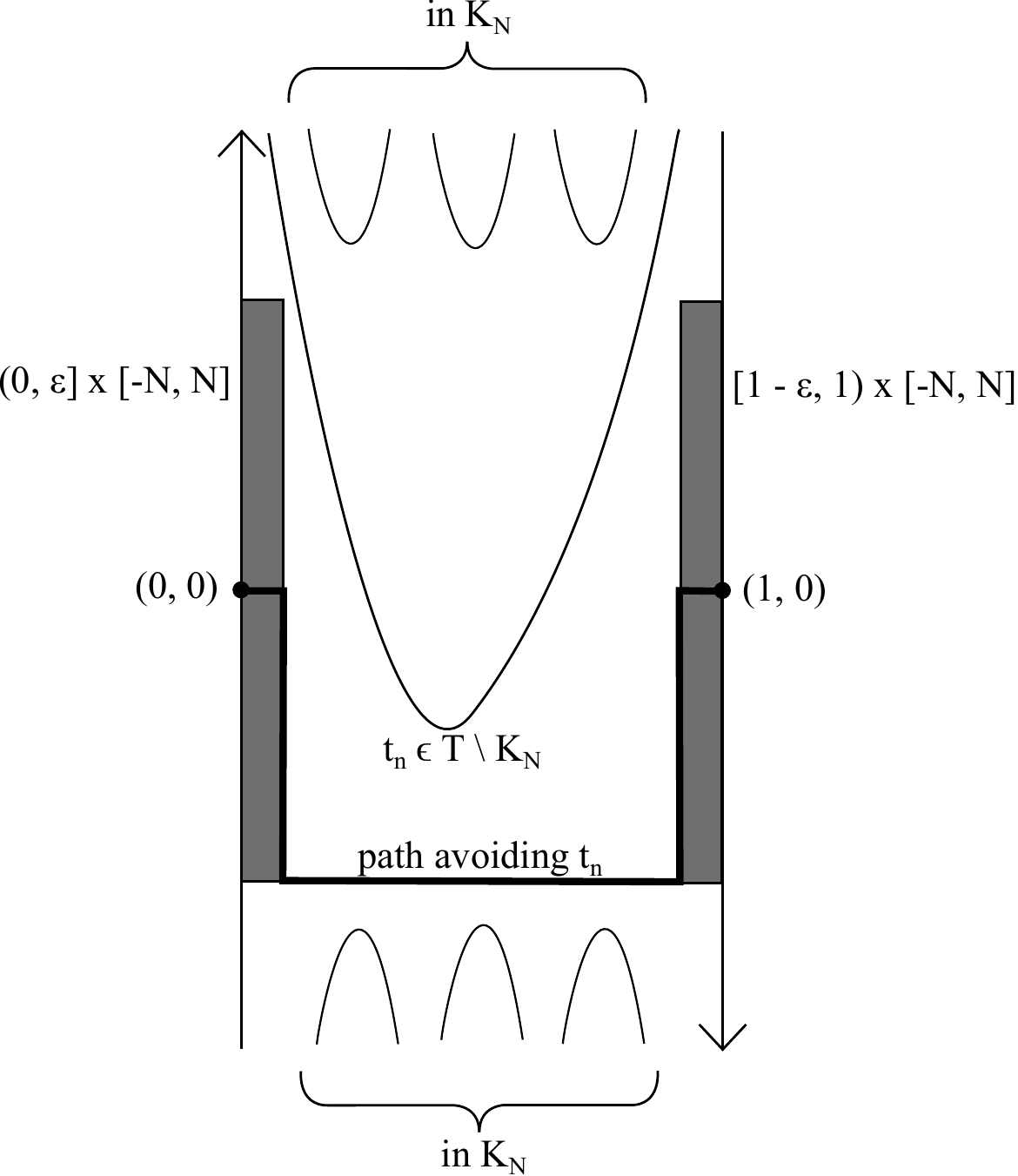}
\caption{Proof of Lemma \ref{NoCrossing}}
\end{figure}

\begin{proof}
Assume without loss of generality that $s_i$ goes from the lower frontier of $\bar{U}$ to the upper frontier of $\bar{U}$. Let $\bar{U}_i$ be the connected component of $\bar{U} \setminus (s_1 \cup \ldots \cup s_{i + 1})$ lying between $s_i$ and $s_{i+ 1}$.

The key will be to show that if $p_i, p_{i + 1}$ are points on $s_i, s_{i + 1}$ respectively, then there is a curve $\gamma\colon [0, 1] \to \A$ such that $\gamma(0) = p_i$, $\gamma(1) = p_{i + 1}$, and $\gamma((0, 1)) \subset \bar{U}_i$.

By the Schoenflies-Homma Theorem, there is a homeomorphism from $\bar{U}_i \cup s_i \cup s_{i + 1}$ to $[0, 1] \times \R$ which is orientation-preserving (with the standard orientation on $\R^2$ and the orientation it induces on $\A = \R^2/h$) and sends $s_i$ to $\{0\} \times \R$, oriented upward, and $s_{i + 1}$ to $\{1\} \times \R$. Using these coordinates, without loss of generality, we are trying to find a curve from $(0, 0)$ to $(1, 0)$ not intersecting $\bar{g}(\overline{x\text{-axis}})$.

First note that for any $N > 0$, there is an $\epsilon > 0$ such that the intersection of $\bar{g}(\overline{x\text{-axis}})$ with $(0, \epsilon] \times [-N, N]$ is empty, and the intersection of $\bar{g}(\overline{x\text{-axis}})$ with $[1 - \epsilon, 1) \times [-N, N]$ is empty. This is by the compactness of $\bar{g}(\overline{x\text{-axis}})$. We may take the beginning of our curve to be a horizontal line from $(0, 0)$ to $(\epsilon, 0)$, and the end to be a horizontal line from $(1 - \epsilon, 0)$ to $(1, 0)$. We must therefore find a curve from $(\epsilon, 0)$ to $(1 - \epsilon, 0)$ that avoids $\bar{g}(\overline{x\text{-axis}})$.


Now there is some countable collection $T = \{t_n\}$ of connected components of $\bar{g}(\overline{x\text{-axis}}) \cap \bar{U}_i$. They all start and end at positive infinity or negative infinity. For any $N > 0$, let $K_N = \{t_n \in T \colon t_n \subset (0, 1) \times (N, \infty) \text{ or } t_n \subset (0, 1) \times (-\infty, -N)\}$. Note that $K_N$ contains all but finitely many of the $t_n$, for the same reasons that (as we argued above) $\bar{g}(\overline{x\text{-axis}})$ can cross $\bar{U}$ only finitely many times. Choose $N$ to be large enough so that every $t_n$ starting and ending at positive infinity is contained in $(0, 1) \times (-N, \infty)$, and every $t_n$ starting and ending at negative infinity is contained in $(0, 1) \times (-\infty, N)$.

Notice that no $t_n \in T \setminus K_N$ disconnects $(\epsilon, 0)$ from $(1 - \epsilon, 0)$ in the open set $(0, 1) \times (-N, N)$. For suppose without loss of generality that $t_n$ starts and ends at positive infinity. Then we can take a curve from $(\epsilon, 0)$ to $(1 - \epsilon, 0)$ that consists of straight lines from $(\epsilon, 0)$ to $(\epsilon, -N)$, from $(\epsilon, -N)$ to $(1 - \epsilon, -N)$, and from $(1 - \epsilon, -N)$ to $(1 - \epsilon, 0)$. Such a path will not cross $t_n$, by the choice of $N$ and $\epsilon$. Since no single $t_n \in T \setminus K_N$ disconnects $(\epsilon, 0)$ from $(1 - \epsilon, 0)$ in $(0, 1) \times (-N, N)$, applying the Schoenflies-Homma Theorem again, their union does not either.

Thus, there is a path from $(\epsilon, 0)$ to $(1 - \epsilon, 0)$ that avoids all $t_n \in T$, i.e., avoids $\bar{g}(\overline{x\text{-axis}})$. Combining this with the horizontal paths from $(0, 0)$ to $(\epsilon, 0)$ and from $(1 - \epsilon, 0)$ to $(1, 0)$, we get our desired curve $\gamma$ from $(0, 0)$ to $(1, 0)$.

Now observe that $\gamma$ leaves $\{0\} \times \R$ (oriented upward) pointing to the right. Since $\bar{g}$ is orientation-preserving, $\bar{g}^{-1}(\gamma)$ leaves $\overline{x\text{-axis}}$ (oriented rightward) also pointing to the right; that is, downward. Since $\gamma$ does not intersect $\bar{g}(\overline{x\text{-axis}})$ except at its endpoints, $\gamma((0, 1))$ lies below $\bar{g}(\overline{x\text{-axis}})$. Since as $t \to 1$, $\bar{g}^{-1}(\gamma(t))$ approaches $\overline{x\text{-axis}}$ from below, by the same reasoning the line $\{1\} \times \R$ in our picture must be oriented downward. That is, the crossing $s_{i + 1}$ goes from the upper frontier of $\bar{U}$ to the lower frontier of $\bar{U}$.

Any curve crossing from the upper frontier to the lower frontier, or vice versa, in $\bar{U}_i$ must cross $\gamma$. Therefore, such a curve must at some point go below $\bar{g}(\overline{x\text{-axis}})$, so if $c$ lies above $\overline{x\text{-axis}}$, $\bar{g}(c)$ cannot cross $\bar{U}$ between $s_i$ and $s_{i + 1}$.

\end{proof}

\begin{lemma}
\label{SameComponent}
Assume the crossing $s_i$ goes from the lower frontier of $\bar{U}$ to the upper frontier (and $s_{i + 1}$ goes from the upper frontier to the lower frontier). The crossings $s_i$ and $s_{i + 1}$ are in the same connected component of $\bar{g}(\bar{U}) \cap \bar{U}$ if and only if, for some essential curve $c \subset \bar{U}$ lying below $\overline{x\text{-axis}}$, $\bar{g}(c)$ does not cross from one frontier to the other in $\bar{U}_i$. Furthermore, if $s_i$ and $s_{i + 1}$ are not in the same connected component, then for every $c$ below $\overline{x\text{-axis}}$, the component of $\bar{g}(\bar{U}) \cap \bar{U}$ containing $s_i$ also contains a crossing of $\bar{g}(c)$, as does the component of $\bar{g}(\bar{U}) \cap \bar{U}$ containing $s_{i + 1}$.

The same facts hold \emph{mutatis mutandis}, switching ``lower'' and ``upper'' and replacing ``below'' with ``above.''
\end{lemma}

\begin{proof}
Suppose $s_i$ and $s_{i + 1}$ are in the same connected component. Then there is a curve $\gamma$ from $s_i$ to $s_{i + 1}$ such that $\gamma \subset \bar{g}(\bar{U}) \cap \bar{U}$. Since $\bar{g}^{-1}(\gamma)$ is compact, starts and ends in $\overline{x\text{-axis}}$, and is contained in $\bar{U}$, there is an essential curve $c \subset \bar{U}$ which is low enough that it does not intersect $\bar{g}^{-1}(\gamma)$. Thus, $\bar{g}(c)$ cannot cross from one frontier to the other of $\bar{U}_i$.

Now suppose that there is an essential curve $c$ below $\overline{x\text{-axis}}$ such that $\bar{g}(c)$ does not cross $\bar{U}_i$. By Lemma \ref{NoCrossing}, we know there is a curve $\gamma$ from $s_i$ to $s_{i + 1}$ whose interior does not intersect $\bar{g}(\overline{x\text{-axis}})$. By the same reasoning, now also taking into consideration $\bar{g}(c) \cap \bar{U}_i$, we can assume that in addition $\gamma \cap \bar{g}(c) = \emptyset$. Therefore, $\gamma$ lies between $\bar{g}(\overline{x\text{-axis}})$ and $\bar{g}(c)$; in particular, it is contained in $\bar{g}(\bar{U})$. Since $\gamma$ is also contained in $\bar{U}$, $s_i$ and $s_{i + 1}$ are in the same connected component of $\bar{g}(\bar{U}) \cap \bar{U}$.

Finally, assume $s_i$ and $s_{i + 1}$ are not in the same connected component of $\bar{g}(\bar{U}) \cap \bar{U}$. Let $c$ be an essential curve in $\bar{U}$ lying below $\overline{x\text{-axis}}$. Let $t$ be the crossing of $\bar{g}(c)$ between $s_i$ and $s_{i + 1}$ that is closest to $s_i$. Since neither $\bar{g}(\overline{x\text{-axis}})$ nor $\bar{g}(c)$ cross $\bar{U}$ between $s_i$ and $t$, applying the reasoning of the above paragraph we conclude that $s_i$ and $t$ are in the same connected component of $\bar{g}(\bar{U}) \cap \bar{U}$.
\end{proof}

Recall we denoted by $\bar{V}_i$ the connected components of $\bar{g}(\bar{U}) \cap \bar{U}$ containing at least one crossing $s_j$ of $\bar{g}(\overline{x\text{-axis}})$.

\begin{lemma}
Some $\bar{V}_i$ has the property that for every essential curve $c \subset \bar{U}$, $\bar{g}(c)$ crosses $\bar{U}$ in $\bar{V}_i$.
\end{lemma}

\begin{proof}
We claim that it is impossible for every $\bar{V}_i$ to contain an even number of crossings. For, suppose otherwise. Suppose, without loss of generality, that $s_1$ goes from the lower frontier of $\bar{U}$ to the upper frontier.

Foliate $\bar{U}$ by essential curves $c_t, t \in \R$, such that $c_0 = \overline{x\text{-axis}}$. By Lemma \ref{NoCrossing}, for $t < 0$, $c_t$ can only cross $\bar{U}$ between $s_j$ and $s_{j + 1}$ if $j$ is odd. Such $s_j$, $s_{j + 1}$ must lie in the same $\bar{V}_i$. Hence, by Lemma \ref{SameComponent}, there is some $t_j < 0$ low enough so that $c_{t_j}$ does not cross $\bar{U}$ between $s_j$ and $s_{j + 1}$. Letting $T$ be the minimum of all these $t_j$, $c_T$ cannot cross $\bar{U}$ at all. This is a contradiction.

Therefore, some $\bar{V}_i$ contains an odd number of crossings. If $s_j$ is the first of these and $s_k$ is the last, then without loss of generality $s_j$ and $s_k$ both go from the lower frontier of $\bar{U}$ to the upper frontier. Since $s_{j - 1}$ is not in the same connected component of $\bar{g}(\bar{U})$ as $s_j$, and $s_{j - 1}$ is oriented downward while $s_j$ is oriented upward, by Lemma \ref{SameComponent}, for every $t > 0$, $\bar{g}(c_t)$ has a crossing of $\bar{U}$ between $s_{j - 1}$ and $s_j$ that lies in $\bar{V}_i$. Similarly, since $s_k$ and $s_{k + 1}$ are not in the same connected component of $\bar{g}(\bar{U}) \cap \bar{U}$, and $s_k$ is oriented upward while $s_{k + 1}$ is oriented downward, by Lemma \ref{SameComponent}, for every $t < 0$, $\bar{g}(c_t)$ has a crossing of $\bar{U}$ between $s_{k}$ and $s_{k + 1}$ that lies in $\bar{V}_i$. Thus, every essential curve in $\bar{U}$ has a $\bar{g}$-image that crosses $\bar{U}$ in $\bar{V}_i$.
\end{proof}

For every $n$, $\bar{f}^n(\bar{V}_i)$ is a component of $\bar{U} \cap \bar{g}(\bar{U})$ containing a crossing of $\bar{g}(\overline{x\text{-axis}})$, i.e., is some $\bar{V}_j$. Since there are only finitely many of these, there exists an $n > 0$ such that $\bar{f}^n(\bar{V}_i) = \bar{V}_i$.

Letting $V$ be a connected component of $U \cap g(U)$ projecting to $\bar{V}_i$, this implies that there is some $m$ such that $f^n(V) = h^m(V)$. Therefore, $\hat{f}$ must have translation number $m/n$ on both the upper boundary line of $\hat{U}$ and on the upper boundary line of $\widehat{g(U)}$. On the other hand, if $\hat{f}$ has translation number $m/n$ on the upper boundary line of $\hat{U}$ then $\widehat{gfg^{-1}} = \widehat{h^{-1}f}$ has the same translation number on the upper boundary line of $\widehat{g(U)},$ so $\hat{f}$ must have translation number $m/n + 1$ on this line, a contradiction.


\begin{thebibliography}{9}
\bibitem{A&S} L. Ahlfors and L. Sario, \emph{Riemann Surfaces}, Vol. 960, Princeton University Press, 1960.

\bibitem{AFW} M. Aschenbrenner, S. Friedl, and H. Wilton, \emph{3-manifold groups}, \newblock{\tt http://arxiv.org/abs/1205.0202}.


\bibitem{Franks2} J. Franks, \emph{Geodesics on $S^2$ and periodic points of annulus homeomorphisms}, Invent. Math. \textbf{108} (1992), No. 1, 403-418.

\bibitem{Franks3} \line(1,0){30}, \emph{Realizing rotation vectors for torus homeomorphisms}, Trans. Amer. Math. Soc. \textbf{311} (1989), No. 1, 107-115.

\bibitem{Franks} \line(1,0){30}, \emph{Rotation vectors and fixed points of area preserving surface diffeomorphisms}, Trans. Amer. Math. Soc. \textbf{348} (1996), No. 7, 2637-2662.

\bibitem{Ghys}
\'{E}. Ghys, \emph{Groups acting on the circle}, Enseign. Math. \textbf{47} (2001), No. 3/4, 329-408.


\bibitem{Homma} T. Homma, \emph{An extension of the Jordan curve theorem}, Yokohama Math. J., \textbf{1} (1953), 125-129.

\bibitem{Mather} J. N. Mather, \emph{Topological proofs of some purely topological consequences of Carath\'{e}odory's theory of prime ends}, Selected studies: physics-astrophysics, mathematics, history of science: a volume dedicated to the memory of Albert Einstein (Th. M. Rassias and G. M. Rassias, eds.), North-Holland Pub. Co., Amsterdam, 1982, pp. 225-255.


\bibitem{Parkhe} K. Parkhe, \emph{Actions of the Heisenberg group on surfaces}, PhD Diss., Northwestern University (2013).

\bibitem{Poincare} H. Poincar\'{e}, \emph{Sur les courbes d\'{e}finies par les \'{e}quations diff\'{e}rentielles}, J. Math. Pures Appl. s\'{e}rie 4 \textbf{1} (1885), 167-244.

\bibitem{Pollicott} M. Pollicott, \emph{Rotation sets for homeomorphisms and homology}, Trans. Amer. Math. Soc. \textbf{331} (1992), No. 2, 881-894.

\bibitem{Rhodes} F. Rhodes, \emph{Asymptotic cycles for continuous curves on geodesic spaces}, J. London Math. Soc. (2), No. 2 (1973), 247-255.

\bibitem{Schwartzman} S. Schwartzman, \emph{Asymptotic Cycles}, Annals of Math. \textbf{66} (1957), No. 2, 270-284.

\bibitem{Wang} J. Wang, \emph{A generalization of the line translation theorem}, \newblock{\tt http://arxiv.org/abs/1104.5185}.
\end{thebibliography}
\end{document}